%% file: main.tex
\documentclass{amsproc}

\usepackage{graphicx}
\usepackage[colorinlistoftodos,textsize=footnotesize]{todonotes}
\usepackage[colorlinks=true, allcolors=blue]{hyperref}
\usepackage[noabbrev,capitalize]{cleveref}
\usepackage{enumitem}
\usepackage{float}
\usepackage{booktabs}
\usepackage{tabulary}
\usepackage{pgfplots}
\pgfplotsset{compat=1.16}

\newcommand{\A}{\mathcal{A}}
\renewcommand{\S}{\mathcal{S}}
\newcommand{\R}{\mathbb{R}}
\newcommand{\Z}{\mathbb{Z}}

\newcommand{\cl}{Q}  

\newtheorem{theorem}{Theorem}[section]
\newtheorem{lemma}[theorem]{Lemma}
\newtheorem{corollary}[theorem]{Corollary} 

\theoremstyle{definition}
\newtheorem{definition}[theorem]{Definition}

\newtheorem{question}[theorem]{Question} 
\newtheorem{conjecture}[theorem]{Conjecture} 

\theoremstyle{remark}
\newtheorem{remark}[theorem]{Remark}

\numberwithin{equation}{section}

\begin{document}

\title{ Piercing Numbers in Circular Societies }


\author{Kristen Mazur}
\address{Elon University}
\curraddr{}
\email{kmazur@elon.edu}
\thanks{}

\author{Mutiara Sondjaja}
\address{New York University}
\email{sondjaja@nyu.edu}
\thanks{}

\author{Matthew Wright}
\address{St.\ Olaf College}
\curraddr{}
\email{wright5@stolaf.edu}
\thanks{}

\author{Carolyn Yarnall}
\address{CSU Dominguez Hills}
\curraddr{}
\email{cyarnall@csudh.edu}
\thanks{}


\dedicatory{ }


\date{}

\begin{abstract}
 In the system of approval voting, individuals vote for all candidates they find acceptable. Many approval voting situations can be modeled geometrically, and thus geometric concepts such as the piercing number  have a natural interpretation. In this paper, we explore piercing numbers in the setting where voter preferences can be modeled by congruent arcs on a circle -- i.e., in fixed-length circular societies. Given a number of voters and the length of the voter preference arcs, we give bounds on the possible piercing number of the society. Further, we explore which piercing numbers are more likely. Specifically, under the assumption of uniformly distributed voter preference arcs, we determine the probability distribution of the piercing number of  societies in which the length of the arcs is sufficiently small. We end with simulations that give estimated probabilities of piercing number for societies with larger voter preference arcs. 
 

\end{abstract}

\maketitle


\section{Introduction}
Approval voting is a system in which each individual casts a vote for all candidates of which they approve, and the winning candidate is the one that receives the most votes. This system is of interest in that it has many advantages over the more common plurality voting systems and it can be applied outside of the political arena \cite{BramsFishburn}. For instance, approval voting is often the way that committees approach scheduling problems: each committee member indicates the times they are available, and the meeting is held at the time most members are able to meet. This situation can be modeled mathematically by considering each person's availability as an interval on a line or an arc on a circle, the latter corresponding to a 24-hour clock which may be especially useful for scheduling virtual meetings occurring across several time zones.

Previous research (\cite{BergEtal}, \cite{Hardin}, \cite{MazurEtal}) developed bounds on the maximum number of voters who approve of a single candidate in various approval voting scenarios. For example, if we model the scheduling problem above using intervals on a line, then Berg et al.\ showed that if two out of every three committee members agree on a meeting time then there is a time that works for at least half of the committee members \cite{BergEtal}. Alternatively, if we model the problem using arcs on a circle, then Hardin showed that if two out every three members agree on a meeting time, then there is a time that works for at least one third of the committee members \cite{Hardin}.

This paper explores a different question. Instead of the maximum number of voters who agree on a single candidate, we develop a bound on the number of candidates needed to satisfy all voters.  Specifically, expanding upon the scheduling problem above, suppose the committee requires all of its members to attend a training session. They plan to hold the same training session at multiple times  so that all members will be able to attend one of the sessions. After viewing each member's availability the committee then asks, ``what is the minimum number of training sessions needed so that each member can attend at least one of the sessions?'' In \cite{SuZerbib}, Su and Zerbib call this collection of training sessions a \textit{representative candidate set} or, using language from convex geometry, a \textit{piercing set}. In this paper we use the latter terminology. More formally, given a collection of sets, a \textit{piercing set} is a collection of points such that each set contains at least one of the points. The size of the smallest possible piercing set is the \textit{piercing number} of the collection of sets. Thus, if we model the above scenario geometrically, then the collection of time intervals at which the members are available is a collection of sets, and the minimum number of training sessions needed is the piercing number.

Su and Zerbib recently contextualized piercing numbers in terms of approval voting \cite{SuZerbib}. They present a collection of piercing results that best apply to approval voting scenarios that are modeled using intervals on a line. Most notably, they cite Hadwiger and Debrunner’s  \cite{HadwigerDebrunner1957} classic 1957 result that provides an upper bound on the piercing number of convex sets in $\R^d$ based on a local intersection property and state its implications in approval voting. Su and Zerbib then leverage this result to make statements regarding piercing sets and piercing numbers in various approval voting scenarios. For example, in a scenario modeled with a collection of intervals on the real number line, if two out of every four voters agree on a candidate, then the piercing number is at most three. Thus, there is a piercing set that contains three candidates.

Such a result does not apply to the above training session scenario because this scenario is best modeled using arcs on a circle (instead of intervals on the real line), to account for the periodic nature of the calendar year or a clock (e.g., members are able to select an availability arc of December--January). In approval voting we call this model a \textit{circular society}. The piercing number for circular societies has scarcely been explored (see Final Remark 1 in \cite{SuZerbib}).

This paper explores piercing numbers of circular societies.
After reviewing the mathematical background on approval voting in \Cref{background}, we provide a complete characterization of possible piercing numbers in circular societies in which all arcs have the same fixed length in \Cref{fixedLength}.  Then, in \Cref{probability}, we explore the probability of certain piercing numbers in circular societies when arcs of a fixed length are randomly distributed around the circle.

\section{Background}\label{background}

Any approval voting scenario consists of a collection of voters and a collection of candidates that voters may approve of. We refer to the set of  candidates  as the \textit{spectrum} $X$. The set of candidates acceptable to an individual voter is called an \textit{approval set}. Thus, each voter has an approval set that is a subset of the spectrum. 

\begin{definition}\label{Society}We define a \textit{society} $\mathcal{S}$ of $n$ voters to be a pair $(X,\A)$ in which $X$ is a spectrum and $\A$ is a collection of $n$ approval sets.
\end{definition}

Despite the term ``candidate'' seemingly implying discreteness, we study societies in which the spectra are continuous. In \cite{BergEtal}, Berg et al.\ introduced the language of a society,  focusing on societies in which the spectrum is the real number line and approval sets are intervals. These societies are called \textit{linear societies}. Alternatively, in a \textit{circular society}, the spectrum is a circle and approval sets are arcs on the circle. Circular societies were first studied by Hardin \cite{Hardin}. The authors of this paper previously studied products of linear and circular societies \cite{MazurEtal}.

Given a society $\mathcal{S}=(X,\A)$, since $\A$ is a collection of sets, we can consider a piercing set of $\A$ or the piercing number of $\A.$ In the context of an approval voting society $\mathcal{S}$, a piercing set of $\A$ is a collection of candidates in $X$ such that each voter approves of at least one candidate in the collection.  The piercing number is the smallest number of candidates needed so that every voter is happy. Thus, we refer to the piercing number of $\A$ as the piercing number of the society $\mathcal{S}$ and denote it $\tau(\mathcal{S})$. Moreover, if an approval set $A$ contains a candidate $x$ we say that $x$ \textit{pierces} $A$. 

Su and Zerbib \cite{SuZerbib} use the notion of $(k,m)$-agreeability to give upper bounds on the piercing numbers of linear and circular societies. A society is $(k,m)$-agreeable \cite{BergEtal} if $k$ out of every $m$ approval sets intersect. For example, in a $(2,3)$-agreeable society, two out of every three voters agree on a candidate. Su and Zerbib’s result, given in Theorem \ref{PierceNumLinSoc} below, follows directly from a theorem  of Hadwiger and Debrunner \cite{HadwigerDebrunner1957}  that gives conditions of overlap for families of convex sets. (See Theorem 5 in \cite{SuZerbib}.)  

\begin{theorem}[Piercing Numbers in Linear Societies \cite{SuZerbib}]\label{PierceNumLinSoc} If $\mathcal{S}$ is a $(k,m)$-agreeable linear society then  we can find a piercing set that contains $m-k+1$ candidates. Hence, $\tau(\mathcal{S}) \le m-k+1$.
\end{theorem}

To obtain an analogous result for a circular society $\mathcal{S}$, we can arbitrarily remove a candidate on the spectrum, cut the circle at that point and ``unroll'' it to produce a linear society. Applying \cref{PierceNumLinSoc} and including the removed candidate as a piercing point, we obtain an upper bound for $\tau(\mathcal{S})$.

\begin{corollary}[Piercing Numbers in Circular Societies \cite{SuZerbib}]\label{circularpiercingnum1}
Suppose $\mathcal{S}$ is a $(k,m)$-agreeable circular society. Then $\mathcal{S}$ has a piercing set that contains $m-k+2$ candidates. Hence, $\tau(\mathcal{S}) \le m-k+2$.
\end{corollary}

While \cref{PierceNumLinSoc} and \cref{circularpiercingnum1} are mathematically interesting, results that give information about piercing based on $(k,m)$-agreeability are difficult to apply to approval voting in practice. In an approval voting scenario, in order to determine if the society is $(k,m)$-agreeable we must know all voter preferences and how they interact with one another. Therefore, we would likely also be able to determine the piercing number. 

In this paper, we study piercing numbers through a different lens. Rather than working with $(k,m)$-agreeable circular societies, we shift our focus to circular societies in which all approval sets have the same length. For example, in a committee scheduling scenario, we could require that all members choose a 3-hour time block during which they are available.

\begin{definition}\label{FixedLength} Let $\mathcal{S}=(X,\A)$ be a circular society. We say that $\mathcal{S}$ is a \textit{fixed-length} circular society if all approval sets in $\A$ have the same length. We call this fixed length the \textit{approval set length} of $\mathcal{S}$ and denote it $p$.
\end{definition}

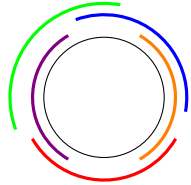
\begin{figure}[tb]
  \begin{center}
    \begin{tikzpicture}
        \draw[] (0,0) circle (0.8cm);
        \draw[orange, very thick] (-60:.95) arc (-60:60:.95cm);
        \draw[violet, very thick] (120:.95) arc (120:240:.95cm);
        \draw[blue, very thick] (-10:1.1) arc (-10:110:1.1cm);
        \draw[green, very thick] (80:1.25) arc (80:200:1.25cm);
        \draw[red, very thick] (210:1.1) arc (210:330:1.1cm);
    \end{tikzpicture}
  \end{center}
  \caption{A fixed-length circular society with five voters. The black circle is the spectrum, and each colored arc represents the approval set of one voter.}
  \label{fig:fiveVoters}
\end{figure}

In order to work with such societies more clearly and concisely, we lay out some useful conventions for circular societies that we use throughout this paper.

\begin{remark}[Circular Societies as $\R/\Z$]\label{RModZ}  We use the classic topological property that a circle is equivalent to $\R/\Z$, identifying the points $x$ and $x+n$ for any real number $x$ and integer $n$. In other words, given a society $\mathcal{S} = (X,\A)$, we represent the circular spectrum $X$ as the interval $[0,1)$, identifying $0$ and $1$. In doing so, we force the circumference of $X$ to be $1$, rescaling all approval sets as needed. Thus, in a fixed-length circular society, the approval set length $p$ is less than 1 and can be thought of as a proportion of the circumference of the spectrum. For example, suppose we are modeling time preferences on a 24-hour clock. If we require individual approval sets to be 3-hour time windows, then since the spectrum circumference is 1, the approval set length $p$ is $\frac{1}{8}$.  \end{remark}

 \begin{remark}[Interval Notation of Approval Sets]\label{IntervalNotation} Following Remark \ref{RModZ}, identifying the spectrum with $\R/\Z$ allows us to depict fixed-length circular societies ``linearly'' with approval sets as intervals (or a union of two intervals if the corresponding arc crosses the point corresponding to $0$) of length $p$, as in \cref{fig:FixedLengthEx}. We also assume that approval sets are closed unless otherwise noted. 

Further, we often label an approval set as $A=[\ell,r]$, referring to $\ell$ as the ``left'' endpoint and $r$ as the ``right'' endpoint. This might seem a bit backwards when considering intervals that contain the point $0=1$. For example, in Figure 4, the left endpoint of the brown interval labeled $A$ is to the right of the right endpoint in the linear depiction. Alternatively, we can think of $\ell$ as the endpoint that is counterclockwise from the center of the interval and $r$ as the endpoint that is clockwise from the center. 
\end{remark}

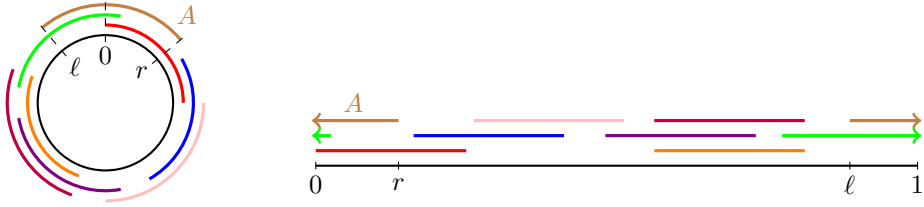
\begin{figure}[ht]
\begin{center}
\begin{tikzpicture}[scale=.9]
    \draw[thick] (0,0) circle [radius=1];
	\draw[red, very thick] ([shift=(0:1.15)]0,0) arc (0:90:1.15);
	\draw[orange, very thick] ([shift=(160:1.15)]0,0) arc (160:250:1.15);
	\draw[green, very thick] ([shift=(80:1.3)]0,0) arc (80:170:1.3);
	\draw[blue, very thick] ([shift=(300:1.3)]0,0) arc (300:390:1.3);
	\draw[violet, very thick] ([shift=(190:1.3)]0,0) arc (190:280:1.3);
	\draw[brown, very thick] ([shift=(40:1.45)]0,0) arc (40:130:1.45);
	\draw[pink, very thick] ([shift=(270:1.45)]0,0) arc (270:360:1.45);
    \draw[purple, very thick] ([shift=(160:1.45)]0,0) arc(160:250:1.45);
    \draw[dashed] (0,0.9)--(0,1.6);
    \node at (0,0.7) {$0$};
    \draw[dashed] (40:0.9)--(40:1.6);
    \node at (40:0.7) {$r$};
    \draw[dashed] (130:0.9)--(130:1.6);
    \node at (130:0.7) {$\ell$};
    \node[brown] at (1.2,1.3) {$A$};
\end{tikzpicture} 
\hspace{1cm}
\begin{tikzpicture}
    \draw[thick] (0,0) -- (8,0);
    \draw (0,0.05) -- (0,-0.05);
    \node at (0,-0.25) {$0$};
    \draw (8,0.05) -- (8,-0.05);
    \node at (8,-0.25) {$1$};
    
    \draw[blue, very thick] (1.3,0.4) -- (3.3, 0.4);
    \draw[pink, very thick] (2.1,0.6)--(4.1, 0.6);
    \draw[violet, very thick] (3.85,0.4)--(5.85,0.4);
    \draw[orange, very thick] (4.5,0.2)--(6.5,0.2);
    \draw[green, very thick,->] (6.2,0.4) -- (8.05,0.4);
    \draw[green, very thick, <-] (-0.05,0.4)--(0.2,0.4);
    \draw[red, very thick] (0,0.2) -- (2,0.2);
    \draw[brown, very thick, ->] (7.1, 0.6)--(8.05,0.6);
    \draw[brown, very thick, <-] (-0.05,0.6)--(1.1,0.6);
    \draw[purple, very thick] (4.5, 0.6) -- (6.5, 0.6);
    
    \node[brown] at (0.5,0.83) {$A$};
    \draw (7.1,0.05) -- (7.1,-0.05);
    \node at (7.1,-0.25) {$\ell$};
    \draw (1.1,0.05) -- (1.1,-0.05);
    \node at (1.1,-0.25) {$r$};
\end{tikzpicture}
\end{center}
\caption{Two depictions of the same fixed-length circular society with approval set length $p=\frac{1}{4}$. Endpoints of approval set $A$ are labeled. We cut the figure on the left at the dashed line and ``unroll'' to obtain the figure on the right. The arrows on two approval sets on the right represent that these approval sets ``wrap around.'' }
\label{fig:FixedLengthEx}
\end{figure}


\section{Possible Piercing Numbers}\label{fixedLength}

Piercing numbers depend on both the size of the approval sets and the number of voters.
Intuitively societies with shorter approval sets may require more points to pierce all sets. For example, consider societies with five voters. If the approval set length is very small (more precisely, less than $\frac{1}{5}$), then the piercing number can range from $1$ to $5$, depending on how the approval sets overlap. However, if the approval set length is very large (greater than $\frac{4}{5}$), then the piercing number must be $1$; this is because the complements of all approval sets cannot cover the entire circle, so there is a point on the circle that is not in the complement of any approval set, and is thus contained in every approval set. The story becomes murkier when approval set lengths are between $\frac{1}{5}$ and $\frac{4}{5}$. 
For example, in \cref{pnumless} we show that when the approval set length is between $\frac{2}{5}$ and $\frac{4}{5}$, the piercing number is less than or equal to two. However, in \cref{pnumgreater} we construct a society with approval set length $\frac{2}{5}-\epsilon$ in which the piercing number equals three. 

Changing perspective, we can consider what happens when we fix the approval set length and vary the number of voters. In general, a society with more voters may have a larger piercing number. For example, a society with approval set length of $\frac{1}{3}$ and four voters will have a  piercing number of at most two, while a society with five voters and approval set length of $\frac{1}{3}$ can have a piercing number of three, as demonstrated in \cref{fig:fiveVoters}. The reasoning behind the first result is nontrivial; the justification relies on the proof of \cref{pnumless}.

\cref{pnFigure} illustrates our main results in this section, stated formally in \cref{pnumless} and \cref{pnumgreater}. In the figure we see how the maximum possible piercing number for any fixed-length circular society depends on the number of voters $n$ and the approval set length $p$. Further, the figure is  best understood by observing the points where the color changes, drawn as dots in the figure.
These dots give thresholds in the $(n,p)$-space at which the maximum piercing number changes by one.
We make this idea precise as follows. Given any $n \in \mathbb{N}$ and $k \in \{1, 2, \ldots, n\}$, there exists a \emph{critical length}
\begin{equation}\label{clEq}
    \cl(n,k) = \dfrac{\left\lceil \frac{n}{k} \right\rceil - 1}{k \left\lceil \frac{n}{k} \right\rceil - k + 1}
\end{equation}
that gives the threshold $p$ below which piercing number $\tau = k+1$ is possible for a fixed-length circular society of $n$ voters.  The following theorems state the relationship between $Q(n,k)$ and possible piercing numbers.

\begin{figure}[h]
    \centering
    \input{possible_piercing_diagram}
    \caption{Maximum possible piercing numbers for fixed-length circular societies with $n$ voters and approval set length $p$.}
    \label{pnFigure}
\end{figure}
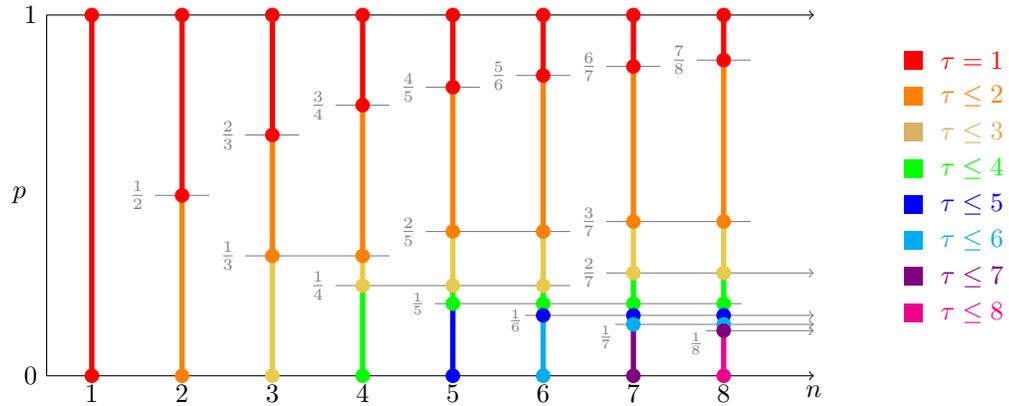

\begin{theorem}\label{pnumless}
    Let $n \in \mathbb{N}$ and $k \in \{1, 2, \ldots, n\}$.
    Then a fixed-length circular society $\mathcal{S}$ with $n$ voters and approval length $p \ge \cl(n,k)$ has piercing number $\tau(\mathcal{S}) \leq k$.
    
\end{theorem}

\begin{theorem}\label{pnumgreater}
    Let $n \in \mathbb{N}$, $k \in \{1, 2, \ldots, n-1\}$, and $0 < p < \cl(n,k)$.
    There exists a fixed-length circular society $\mathcal{S}$ with $n$ voters, approval set length $p$, and piercing number $\tau(\mathcal{S}) = k+1$.
\end{theorem}

    

Before proving these theorems, we offer some comments about $\cl(n,k)$. We find the expression for $\cl(n,k)$ in \cref{clEq} to be non-obvious but descriptive of certain patterns in possible piercing numbers, most notably the stepped asymptotic behavior of the threshold values that appear in \cref{pnFigure}.

Intuitively, with more voters, longer approval sets are needed to guarantee a specific piercing number.
Thus, as the number of voters $n$ increases, the threshold for $p$ below which piercing number $k$ is possible increases as well. That is, for a fixed piercing number $k$, $\cl(n,k)$ increases with $n$, but this increase is strictly monotonic only for $k=1$. 
In general, we find that $\cl(n,k)$ increases in steps of size $k$. \Cref{pnFigure} highlights this by horizontal line segments connecting equal values of $\cl(n,k)$. For example, in a society of $4$ voters, the approval set length must be $p \ge \frac{1}{4}$ in order to guarantee a piercing number of $3$ or less, and this minimum length holds  if a fifth or sixth voter joins the society. However, if a seventh voter joins the society, then the approval set length must be at least $\frac{2}{7}$ to guarantee a piercing number of $3$ or less.


Furthermore, for fixed $k$, $\cl(n,k)$ is bounded above by an asymptote at $\frac{1}{k}$.
Specifically, 
\[ \lim_{n \to \infty} \cl(n,k) = \frac{1}{k}. \]
For example, a piercing number of $1$ is only guaranteed when $p \ge \frac{n-1}{n}$; below this, piercing number $2$ is possible. Hence, since $\lim_{n\to \infty}\frac{n-1}{n}=1$, if a society has a large number of voters then a piercing number of 2 is possible even when the approval set length is very close to $1$. Likewise, as $n \to \infty$ a piercing number of $3$ becomes possible for approval set lengths $p$ approaching $\frac{1}{2}$ from below, yet a piercing number of $3$ is never possible for $p \ge \frac{1}{2}$.
Similar results hold for greater piercing numbers: as $n \to \infty$, a piercing number of $k+1$ is only possible for $p$ approaching $\frac{1}{k}$ from below.

We now prove \cref{pnumless}.

\begin{proof}[Proof of \cref{pnumless}]
   Let $\mathcal{S} = (X,\mathcal{A})$ be a fixed-length circular society with $n$ voters and approval set length $p\ge Q(n,k)$. We will show that $\tau(\mathcal{S}) \le k$. To simplify notation, let $q = \left\lceil \frac{n}{k} \right\rceil$, so $\cl(n,k) = \frac{q-1}{kq-(k-1)}$. We proceed in two cases, depending on whether there exists an open arc $C\subseteq X$ of  length $\frac{1}{kq-(k-1)}$  that contains no \emph{left} endpoints of approval sets. (Recall from Remark~\ref{IntervalNotation} that we call the counterclockwise endpoint of an approval set the left endpoint.)

    For our first case, suppose there exists an open arc $C$ of length $\frac{1}{kq-(k-1)}$ that contains no left endpoints of approval sets.
    Choose a point $x_1\in C$ between the left endpoint of $C$ and the leftmost right endpoint of any approval set that intersects $C$, as illustrated on the left in \cref{fig:theoremIllustration}. Thus, $x_1$ pierces all approval sets that intersect $C$.
    If no approval set intersects $C$, then choose $x_1 \in C$ arbitrarily.
    
    Any approval set not pierced by $x_1$ is contained in the complement $\overline{C}$ of $C$.
    The length of $\overline{C}$ is $1 - \frac{1}{kq-(k-1)} = \frac{kq-k}{kq-(k-1)}$.
    There are $k-1$ equally-spaced points $x_2, x_3, \ldots, x_{k}$ in $\overline{C}$ that divide $\overline{C}$ into $k$ intervals of equal length $\frac{q-1}{kq-k+1}$.
    Any approval set of length $p \geq \frac{q-1}{kq-(k-1)}$ must be pierced by at least one of these points. Thus, $\{x_1, x_2, x_3, \ldots, x_{k}\}$ is a piercing set, so $\tau(\mathcal{S}) \le k$.
    
    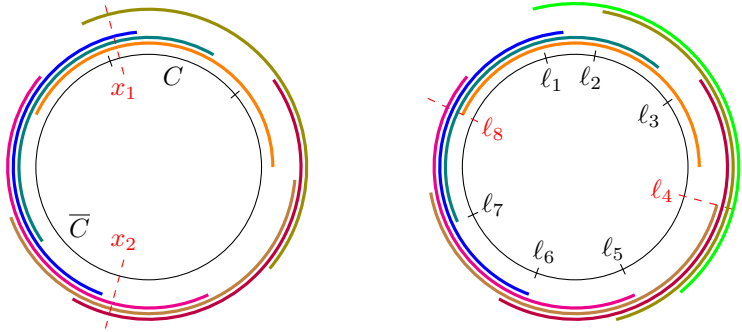
\begin{figure}[ht]
      \centering
      \begin{tikzpicture}[scale=0.75]
        \draw (0,0) circle (2);
        \draw (40:1.9) -- (40:2.1);
        \draw (110:1.9) -- (110:2.1);
        \node at (75:1.68) {$C$};
        \node at (220:1.6) {$\overline{C}$};
        
        \draw[orange, very thick] (0:2.2) arc (0:155:2.2);
        \draw[teal, very thick] (60:2.3) arc (60:215:2.3);
        \draw[blue, very thick] (95:2.4) arc (95:250:2.4);
        \draw[magenta, very thick] (140:2.5) arc (140:295:2.5);
        \draw[brown, very thick] (200:2.6) arc (200:355:2.6);
        \draw[purple, very thick] (240:2.7) arc (240:395:2.7);
        \draw[olive, very thick] (320:2.8) arc (320:475:2.8);

        \draw[dashed,red] (105:1.7) node[below] {$x_1$} -- (105:3);
        \draw[dashed,red] (255:1.7) node[above] {$x_2$} -- (255:3);
      \end{tikzpicture}\hspace{30pt}
      \begin{tikzpicture}[scale=0.75]
        \draw (0,0) circle (2);
        
        \draw[orange, very thick] (0:2.2) arc (0:155:2.2);
        \draw[red,dashed] (155:1.9) -- (155:3);
        \node[red] at (155:1.6) {$\ell_8$};
        
        \draw[teal, very thick] (50:2.3) arc (50:205:2.3);
        \draw (205:1.9) -- (205:2.1);
        \node at (205:1.6) {$\ell_7$};
        
        \draw[blue, very thick] (95:2.4) arc (95:250:2.4);
        \draw (250:1.9) -- (250:2.1);
        \node at (250:1.6) {$\ell_6$};
        
        \draw[magenta, very thick] (140:2.5) arc (140:295:2.5);
        \draw (295:1.9) -- (295:2.1);
        \node at (295:1.6) {$\ell_5$};
        
        \draw[brown, very thick] (190:2.6) arc (190:345:2.6);
        \draw[red,dashed] (345:1.9) -- (345:3);
        \node[red] at (345:1.6) {$\ell_4$};
        
        \draw[purple, very thick] (240:2.7) arc (240:395:2.7);
        \draw (35:1.9) -- (35:2.1);
        \node at (35:1.6) {$\ell_3$};
        
        \draw[olive, very thick] (285:2.8) arc (285:440:2.8);
        \draw (80:1.9) -- (80:2.1);
        \node at (80:1.6) {$\ell_2$};
        
        \draw[green, very thick] (310:2.9) arc (310:465:2.9);
        \draw (105:1.9) -- (105:2.1);
        \node at (105:1.6) {$\ell_1$};
      \end{tikzpicture}
      \caption{In the proof of \cref{pnumless}, either there exists an open arc $C$ of length $\frac{1}{kq-(k-1)}$ containing no left endpoints of approval sets (illustrated at left), or there does not exist such an arc $C$ (illustrated at right).}
      \label{fig:theoremIllustration}
    \end{figure}
    
    For the second case, suppose there does not exist an (open) arc $C\subseteq X$ of length $\frac{1}{kq-(k-1)}$ that contains no \emph{left} endpoints of approval sets.
    This implies $n \ge kq - (k-2)$, for if $n < kq - (k-2)$ then there are too few left endpoints of approval sets in $X$ to avoid a gap of size $\frac{1}{kq-(k-1)}$.
    Label the left endpoints of approval sets consecutively $\ell_1, \ell_2, \ldots, \ell_n$, as illustrated on the right in \cref{fig:theoremIllustration}. 
    
    Consider the points $\ell_q, \ell_{2q}, \ell_{3q}, \ldots, \ell_{(k-1)q}, \ell_n$.
    Since every two consecutive left endpoints of approval sets are separated by an arc of length less than $\frac{1}{kq-(k-1)}$, the arc from $\ell_1$ to $\ell_q$ has length less than $\frac{q-1}{kq-(k-1)}$. 
    Since approval sets have length $p \geq \frac{q-1}{kq-(k-1)}$, this implies that the point $\ell_q$ pierces the approval sets whose left endpoints are $\ell_1, \ell_2, \ldots, \ell_q$.
    Likewise, each point $\ell_{2q}, \ell_{3q}, \ldots, \ell_{(k-1)q}$ pierces $q$ approval sets, and the last point $\ell_n$ pierces the remaining approval sets. Thus, $\{\ell_q, \ell_{2q}, \ell_{3q}, \ldots, \ell_{(k-1)q}, \ell_n\}$ form a piercing set, so $\tau(\mathcal{S}) \le k$.
\end{proof}

Our proof \cref{pnumgreater} relies on the construction of fixed-length circular societies whose approval sets are spaced around the circle as evenly as possible. 
Our construction is similar to the uniform societies defined by Hardin \cite{Hardin}, but in contrast to Hardin, our approval sets are closed rather than half-open. We give the formal definition below.

\begin{definition}
    Let $N$ and $T$ be positive integers such that $N>T$ and let $\epsilon$ be a parameter such that $0 < \epsilon < \frac{1}{N}$. Define the \textit{uniform society} $U_\epsilon(N,T)$ to be the fixed-length circular society  with $N$ voters,  approval set length $\frac{T}{N}-\epsilon$, and approval sets defined as follows.
    Identifying the spectrum with the real numbers modulo $1$ as in \cref{RModZ}, the approval set for voter $i \in \{0, 1, \ldots, N-1\}$ is the interval $A_i = \left[ \frac{i}{N}, \frac{i+T}{N} - \epsilon \right]$. 
\end{definition}

We give an example of a uniform society in \cref{UniformSocPic}.  Note that the combined total length of all approval sets in a uniform society $U_\epsilon(N,T)$ is $N \cdot \left(\frac{T}{N}-\epsilon\right) = T - N\epsilon$. When we assume $\epsilon$ to be small, the total length is approximately $T$.
Furthermore, every point on the spectrum is contained in either $T$ or $T-1$ approval sets.

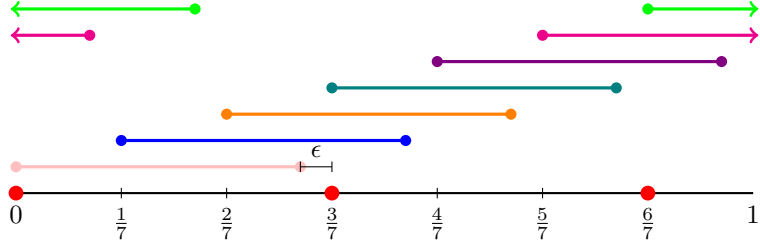
\begin{figure}[h]
\centering
  \begin{tikzpicture}[scale=1.4]
    \draw[thick] (0,0) -- (7,0);
    \node at (0,-0.2) {0};
    \node at (7,-0.2) {1};
    \foreach \x in {1,...,6}:
      \draw (\x, -0.05) node[below] {$\frac{\x}{7}$} -- (\x,0.05);

    \draw[pink, very thick] (0,0.25) -- (2.7,0.25);
    \fill[pink] (0,0.25) circle(1.5pt);
    \fill[pink] (2.7,0.25) circle(1.5pt);
    \draw[blue, very thick] (1, 0.5) -- (3.7, 0.5);
    \fill[blue] (1,0.5) circle(1.5pt);
    \fill[blue] (3.7,0.5) circle(1.5pt);
    \draw[orange, very thick] (2,0.75) -- (4.7,0.75);
    \fill[orange] (2,0.75) circle(1.5pt);
    \fill[orange] (4.7,0.75) circle(1.5pt);
    \draw[teal, very thick] (3,1) -- (5.7,1);
    \fill[teal] (3,1) circle(1.5pt);
    \fill[teal] (5.7,1) circle(1.5pt);
    \draw[violet, very thick] (4,1.25) -- (6.7,1.25);
    \fill[violet] (4,1.25) circle(1.5pt);
    \fill[violet] (6.7,1.25) circle(1.5pt);
    \draw[magenta, very thick, ->] (5,1.5) -- (7.05,1.5);
    \draw[magenta, very thick, <-] (-0.05,1.5) -- (0.7,1.5);
    \fill[magenta] (5,1.5) circle(1.5pt);
    \fill[magenta] (0.7,1.5) circle(1.5pt);
    \draw[green, very thick, ->] (6,1.75) -- (7.05,1.75);
    \draw[green, very thick, <-] (-0.05,1.75) -- (1.7,1.75);
    \fill[green] (6,1.75) circle(1.5pt);
    \fill[green] (1.7,1.75) circle(1.5pt);

    \fill[red] (0,0) circle (2pt);
    \fill[red] (3,0) circle (2pt);
    \fill[red] (6,0) circle (2pt);
    
    \draw (2.7,.2) -- (2.7,.3);
    \draw (3,.2) -- (3,.3);
    \draw (2.7,.25) -- (3,.25);
    \node[above] at (2.85,.25) {$\epsilon$};
    
  \end{tikzpicture}
\caption{The uniform society $U_\epsilon(7,3)$. The highlighted points on the spectrum form the piercing set constructed in \cref{uniformSocietyPiercingNumber}. Hence, the piercing number is three.}
\label{UniformSocPic}
\end{figure}

\begin{lemma}\label{uniformSocietyPiercingNumber}
    The uniform society $U_\epsilon(N,T)$ has piercing number $\left\lceil \frac{N}{T} \right\rceil$.
\end{lemma}
\begin{proof}
    Since every point in the spectrum is contained in at most $T$ approval sets,
    at least $\left\lceil \frac{N}{T} \right\rceil$ points are necessary to pierce all sets.
    
    We claim that the collection 
    \[ W = \left\{0, \frac{T}{N}, 2\frac{T}{N}, 3\frac{T}{N}, \ldots, \left(\left\lceil\frac{N}{T}\right\rceil-1\right)\frac{T}{N} \right\} \]
    is a piercing set. 
    To justify this claim, observe that $W$ consists of $\left\lceil\frac{N}{T}\right\rceil$ points in $[0, 1)$, and the distance between neighboring points in $W$ is $\frac{T}{N}$.
    Since $[0,1) = \mathbb{R}/\mathbb{Z}$, the distance between the points $0$ and $\left(\left\lceil\frac{N}{T}\right\rceil-1\right)\frac{T}{N}$ on the circular spectrum is less than or equal to $\frac{T}{N}$ (for an example, see \Cref{UniformSocPic}). 
    Since each approval set is of length $\frac{T}{N}-\epsilon$ with its left endpoint at a multiple of $\frac{1}{N}$, each approval set must contain a point in $W$. 
    Specifically, piercing point $j\frac{T}{N}$ is contained in $A_i$ for each $i \in \{ jT-(T-1), \ldots, jT-1, jT \}$. Thus, $W$ is a piercing set, and the piercing number is $\left\lceil \frac{N}{T} \right\rceil$.
\end{proof}

We now prove \cref{pnumgreater}.


\begin{proof}[Proof of \cref{pnumgreater}]
    Let $n \in \mathbb{N}$, $k \in \{1, 2, \ldots, n-1\}$, and $0 < p < \cl(n,k)$. We will construct a fixed-length circular society $\mathcal{S}$ with $n$ voters, approval set length $p$, and piercing number $\tau(\mathcal{S}) = k+1$.
    
    First, we construct a uniform society  with piercing number $\tau = k + 1$.
    Let $N = k \left\lceil \frac{n}{k} \right\rceil - k + 1$, the denominator of $Q(n,k)$; $T = \left\lceil \frac{n}{k} \right\rceil - 1$, the numerator of $Q(n,k)$; and $\epsilon = \min\left(\frac{T}{N} - p, \frac{1}{2N}\right)$. 
    Then define the uniform society $\mathcal{U} = U_\epsilon(N,T)$. \cref{uniformSocietyPiercingNumber} implies that $\mathcal{U}$ has piercing number
    \[ \tau = \left\lceil \frac{N}{T} \right\rceil = \left\lceil \frac{k\left(\left\lceil \frac{n}{k} \right\rceil - 1\right) + 1}{\left\lceil \frac{n}{k} \right\rceil-1} \right\rceil = \left\lceil k + \frac{1}{\left\lceil \frac{n}{k} \right\rceil-1} \right\rceil = k+1. \]

    Moreover, $\mathcal{U}$ has $N \le n$ voters. To see this, note that $k \left\lceil \frac{n}{k} \right\rceil$ is the smallest multiple of $k$ that is greater than or equal to $n$, which implies $k \left\lceil \frac{n}{k} \right\rceil - n \le k - 1$. Thus $N = k \left\lceil \frac{n}{k} \right\rceil - k + 1 \le n$. Further, since $\epsilon \le \frac{T}{N} - p$, the approval sets in $\mathcal{U}$ have length $\frac{T}{N}-\epsilon \ge p$. 
    
    
    
    While $\mathcal{U}$ has the desired piercing number, it may not be the society $\S$ that we set out to create because it may not have enough voters or the approval set length may be too big. However, both of these issues are easy to rectify.  If $N < n$, then augment $\mathcal{U}$ with $n-N$ additional approval sets, each given by the interval $[0,p]$ to create a society with $n$ voters and piercing number $k+1$.  If $\frac{T}{N}-\epsilon > p$, shrink each approval set  around its piercing point to obtain a  society with approval sets of length $p$ and piercing number $k+1$. After performing one or both of these operations we have created a society $\mathcal{S}$ as desired.\end{proof}

For example, given $n = 8$ and $k = 2$, following the proof of \cref{pnumgreater}, we can create a society with $8$ voters, an approval set length less than $Q(n,k)=\frac{3}{7}$, and piercing number $\tau= k+1 = 3$.  Computing $N=7$ and $T=3$ as in the proof, we first construct the uniform society $U_\epsilon(7,3)$ that has 7 voters and piercing number $\tau=3$. \Cref{UniformSocPic} illustrates $U_\epsilon(7,3)$ with $p$ slightly less than $\frac{3}{7}$. 
 Adding the interval $[0,p]$ as an eighth approval set does not change the piercing number.
Furthermore, we can modify $U_\epsilon(7,3)$ to make $p$ arbitrarily small while preserving the piercing number $\tau=3$ by contracting each approval set around the piercing point it contains.

\section{Probabilities of Piercing Numbers}\label{probability}

In \Cref{fixedLength}, we explored which piercing numbers are possible for fixed-length circular societies with a given number of voters $n$ and a given approval set length $p$.  For example, because $Q(5, 2) = \frac{2}{5}$, any circular society with $n = 5$ voters and an approval set length $p = \frac{2}{5}$ has a piercing number of at most 2, and we can construct a circular society with $n=5$ voters and approval set length $p = \frac{2}{5} - \epsilon$ (say $p = 0.3999$) that has  piercing number 3. However, the method for constructing such a circular society in the proof of \Cref{pnumgreater} is quite specific. Hence, intuitively, we might believe that most circular societies with $n=5$ and $p=0.3999$ would have a piercing number that is less than the three.


In this section, we develop a probabilistic model for understanding the overall distribution of piercing numbers of fixed-length circular societies. In particular, we explore the space of fixed-length circular societies whose approval sets are randomly-generated with a uniform distribution on the spectrum $X$.  We define these societies below.

\begin{definition}\label{defn:randomsociety}
A \textit{random fixed-length circular society} with $n$ voters and approval set length $p \in (0, 1)$ is a circular society $\S = (X, \A)$ in which $\A$ consists of  $n$ approval sets of the form $A_i = [\ell_i, r_i]$ for $i = 1, \ldots, n$ where
\begin{itemize}
    \item $\ell_1, \ldots, \ell_n$ are independent random variables, uniformly distributed on the spectrum $X$, and
    \item $r_i = \ell_i+ p$.
\end{itemize}
Recall from \Cref{RModZ} that we identify $X$ with $\R / \Z$.  So, $r_i = \ell_i + p \mod 1$.
\end{definition}

Our ultimate goal is to provide a complete probability distribution for the piercing number as a function of the number of voters, $n$, and the approval set length, $p$. In other words, we hope to answer the question below.

\begin{question}\label{question:tauk}
Given a random fixed-length circular society $\S$ with $n$ voters and approval sets of length $p \in (0,1)$, 
and a positive integer $k$, what is the probability that $\tau(\S) = k$?
\end{question}

For arbitrary $n$, we are able to answer this question only if the approval set length $p$ is sufficiently large or sufficiently small. \Cref{pnumless} tells us that if $p \geq Q(n, 1) = \frac{n-1}{n}$, then the piercing number must be 1. If $p$ is sufficiently small, then the  approval sets may be pairwise disjoint, making it possible for the piercing number to equal the number of voters (i.e., $k=n$). The following theorem from Solomon \cite[Equation (4.5)]{solomon} gives us the probability that the piercing number equals the number of voters when $p<\frac{1}{n}$.

\begin{theorem}[Solomon]\label{thm:solomon}
The probability that $n$ arcs of length $p \in [0, 1/n)$, randomly placed on the circle, are pairwise disjoint is $(1-np)^{n-1}$.
\end{theorem}

Determining the probability that the piercing number equals $k$, for $k$ between 1 and the number of voters, is a difficult task. \Cref{randomsociety_tauk} below, which is the main result of this section, provides the probability that the piercing number is $k$ assuming the approval set length is sufficiently small.

\begin{theorem}\label{randomsociety_tauk}
Let $k$ be a positive integer and $\S$ be a random fixed-length circular society with $n$ voters and approval set length $p < \frac{1}{2k}$.  The probability that $\tau(\S) = k$ is
\begin{equation}\label{probEq}
    \binom{n}{k}(1-kp)^{k-1}(kp)^{n-k}.
\end{equation}
\end{theorem}

We prove \Cref{randomsociety_tauk} in \Cref{section:probabilityproofs}.  
While \Cref{question:tauk} largely remains open for general values of $n$, $k$, and $p$, we are able to give a complete answer for some special cases, when $n$ is small.
We present results for these special cases in  \Cref{section:specialcases}.  In \Cref{section:simulationresults}, we discuss simulation results, including those that suggest that \Cref{randomsociety_tauk} might apply for a greater range of values of $p$.  Throughout, we use the observations summarized in the remark below.

\begin{remark}\label{RightLeft} 
Let $\S = (X, \A)$ be a random fixed-length circular society with $n$ voters, approval set length $p$, and approval sets $A_i = [\ell_i, r_i]$ as defined in \Cref{defn:randomsociety}.  Since $\ell_1, \ldots, \ell_n$ are independent random variables that are uniform on $\mathbb{R}/ \mathbb{Z} = [0, 1)$, each $\ell_i$ has probability density function $f_{\ell_{i}}(x) = \mathbf{1}_{[0, 1)}(x)$.  This means that
\begin{enumerate}
    \item given any interval $I \subseteq \mathbb{R}/ \mathbb{Z}$, the probability that $\ell_i \in I$ is equal to the length of $I$, and
    \item the approval sets $A_1, \ldots, A_n$ are distinct with probability 1.
\end{enumerate}
\end{remark}

\subsection{Proofs of Piercing Probabilities} \label{section:probabilityproofs}

Our main goal in this section is to prove  \Cref{randomsociety_tauk}, which gives the probability that $\tau(\S)=k$ when the length of approval sets satisfies $p < \frac{1}{2k}$.
The proof relies on \Cref{thm:solomon} and also on \Cref{lemma:disjointkPiercingSets}, below.


\begin{lemma}\label{lemma:disjointkPiercingSets}
Let $\S = (X, \A)$ be a fixed-length circular society with $n$ distinct approval sets of length $p < \frac{1}{2k}$. Then $S$ has a piercing number $\tau(\S) = k$ if and only if there exists a unique collection of $k$ pairwise-disjoint approval sets in $\A$ whose union contains the left endpoints of all other approval sets.
\end{lemma}

\begin{proof}[Proof of \Cref{lemma:disjointkPiercingSets}]
Consider a fixed-length circular society $\S = (X, \A)$ with distinct approval sets $A_1, \ldots, A_n$ of length $p < \frac{1}{2k}$.  We follow \Cref{defn:randomsociety}, denoting $A_i = [\ell_i, r_i]$, where $r_i = \ell_i + p$. 

First, suppose there exists a unique collection of $k$ pairwise-disjoint approval sets $A_1, \ldots, A_k$ whose union contains the left endpoints of all other approval sets.  This means that each of the remaining $n-k$ approval sets contains one of the right endpoints of $A_1, \ldots, A_k$.  Thus, the set of the right endpoints of $A_1, \ldots, A_k$ is a piercing set of size $k$.  Since $A_1, \ldots, A_k$ are pairwise disjoint, a minimum piercing set of $\S$ must contain at least $k$ points.  Therefore, $\tau(\S) = k$.

Conversely, suppose $\tau(\S) = k$.  Then, there is a minimum piercing set $W = \{x_1,\ldots, x_k\}$ of $\S$.  For each $i = 1, \ldots, k$, let $\mathcal{S}_i$ denote the collection of all approval sets pierced by $x_i$.  Let $L_{i}$ denote the leftmost left endpoint among sets in $\mathcal{S}_i$ ($i=1, \ldots, k$) and let $R_{i}$ denote the rightmost right endpoint among sets in $\mathcal{S}_i$.  Since each set in $\mathcal{S}_i$ is pierced by $x_i$,  then $L_{i}$ and $R_{i}$ are each within distance $p$ from $x_i$.  Therefore, $R_{i} - L_{i} \leq 2p < \frac{1}{k}$.  So, the union of all approval sets in each $\mathcal{S}_i$ is a closed interval of length strictly less than $\frac{1}{k}$.

This means that the union of all sets in $\S$ is a union of closed intervals whose total length is strictly less than $k \frac{1}{k} = 1$, which implies that not all points on the circle are covered by the approval sets of $\S$.  Therefore, $\S$ is essentially a linear society: We can choose a point $x$ not contained in any approval set; then, we can ``cut'' the spectrum at $x$ and ``unroll'' it to produce a linear society.  Note that choosing a different point $x$ may have the effect of a cyclic permutation of the approval sets in the $S$, but does not affect the results of this theorem. We treat $\S$ as a linear society in the remainder of the proof.

For general collections of arbitrary sets, the task of finding minimum piercing sets is computationally hard.  However, for linear societies, there is a simple linear-time algorithm that finds a minimum piercing set \cite[Section 1.4]{aamahmood}: Start by ordering the approval sets by  their right endpoints.  Find the approval set with the leftmost right endpoint, $r$, and add $r$ into the piercing set; then, remove all approval sets pierced by $r$.  Repeat with the remaining sets until all sets are pierced.  The collection of all right endpoints chosen in this way is a minimum piercing set.  In addition, the approval sets that correspond to these right endpoints are pairwise disjoint.

Since we can treat $\S$ as a linear society, we can apply this algorithm to $\S$, to produce a minimum piercing set, call it $W'$.  The set $W'$ might be different than the minimum piercing set $W$ we start with, but it must also contain $k$ points.  Further, $W'$ consists of $k$ \emph{right endpoints}, call them $r_1, \ldots, r_k$ (with $r_1 < \ldots < r_k$), the right endpoints of approval sets $A_1 = [\ell_1, r_1], \ldots, A_k = [\ell_k, r_k]$.  The collection of approval sets $\{A_1, \ldots, A_k\}$ consists of $k$ pairwise disjoint sets.

Next, we show that the union, $\displaystyle \cup_{i=1}^k A_i$, contains all other left endpoints: For each $i = 1, \ldots, k$, let $\mathcal{T}_i$ be the collection of approval sets that contain $r_i$ but not $r_j$ for any $j < i$.
Since $A_1, \ldots, A_k$ are pairwise disjoint, then $A_i \in \mathcal{T}_i$.  Since $\S$ is a fixed-length society with distinct approval sets, no approval set is contained in another; therefore, any other approval set in $\mathcal{T}_i$ must have its left endpoint in $A_i$.  Since $r_1, \ldots, r_k$ form a piercing set, every approval set is in one of the $\mathcal{T}_i$.  Thus, $A_1, \ldots, A_k$ collectively contains the left endpoints of all approval sets in the society.

Finally, it remains to show that $\{A_1, \ldots, A_k\}$ is the \textit{unique} subcollection of $k$ pairwise disjoint approval sets that contain the left endpoints of the remaining $n-k$ approval sets. First,  $A_1$ must be the leftmost approval set in $S$, since no other approval set contains its left endpoint.  Likewise, $A_2$ must be the leftmost approval set that is disjoint from $A_1$.   Similarly, $A_i$ must be the leftmost approval set that is disjoint from $A_1 \cup A_2 \cup \cdots \cup A_{i-1}$.
This results in a unique collection of approval sets $A_1, \ldots, A_k$ that contain the left endpoints of all remaining approval sets.


\end{proof}

We are now ready to prove \Cref{randomsociety_tauk}.

\begin{proof}[Proof of \Cref{randomsociety_tauk}]
Let $\S = (X, \A)$ be a random fixed-length circular society with approval sets $A_1$, $\dots$, $A_n$ of length $p<
\frac{1}{2k}$. Denote $A_i = [\ell_i, r_i] = [\ell_i, \ell_i  + p]$.  We can assume that these approval sets are distinct, which occurs with probability 1 (\Cref{RightLeft}).

Let $E$ denote the event that $\tau(\S) =k$.  By \Cref{lemma:disjointkPiercingSets}, $E$ is the event that there exists a unique collection of $k$ pairwise disjoint approval sets whose union contains the left endpoints of the remaining $n-k$ sets. Our goal is to compute $P(E)$, by expressing $E$ as a union of disjoint events.

Suppose $T$ is a subset of $\{1, \ldots, n\}$ of cardinality $k$.  Let $D_{T}$ denote the event that the approval sets $\{A_i\}_{i \in T}$ are pairwise disjoint.  Let $C_{T}$ be the event that $\bigcup_{i \in T}A_i$ contains the left endpoints of the remaining $n-k$ approval sets.  Then, $C_{T} \cap D_{T}$ denotes the event that $\{A_i\}_{i \in T}$ is pairwise disjoint and the union $\bigcup_{i \in T} A_i$ contains that left endpoints of the remaining sets. The uniqueness of the collection of $k$ pairwise disjoint sets specified in \Cref{lemma:disjointkPiercingSets} implies that the events $\left\{C_{T} \cap D_{T}\right\}_{T \subseteq \{1, \ldots, n\}, \ |T| = k}$ are pairwise disjoint.

Thus, $E$ is the union of the disjoint events $\left\{C_{T} \cap D_{T}\right\}_{T \subseteq \{1, \ldots, n\}, \ |T| = k}$, which implies that
\[P(E) = P\left(\bigcup_{T \subseteq \{1, \ldots, n\}, \ |T|=k }  (C_{T}\cap D_{T})\right) = \sum_{T \subseteq \{1, \ldots, n\}, \ |T| = k }  P(C_{T} \cap D_{T}). \]

It remains to compute $P(C_{T} \cap D_{T})= P(D_{T}) P(C_{T} | D_{T})$. 
\Cref{thm:solomon} shows that the probability that $k$ arcs of length $p < \frac{1}{k}$, randomly placed on the circle, are pairwise disjoint is $(1-kp)^{k-1}$.
Thus, $P(D_{T}) = (1-kp)^{k-1}$.
Next, we compute $P(C_{T} | D_{T})$. Given that $k$ approval sets $\{A_i\}_{i \in T}$ are disjoint, the length of $\bigcup_{i \in T} A_i$ is $kp$.  Hence, the probability that the left endpoints of the remaining $n-k$ approval sets are in $\bigcup_{i \in T} A_i$ is $(kp)^{n-k}$. Therefore, $P(C_{T} \cap D_{T}) = P(D_{T}) P(C_{T} | D_{T}) = \left(1-kp\right)^{k-1}(kp)^{n-k}$.  

Since the number of possible combinations of $k$ disjoint approval sets is $\binom{n}{k}$, the probability that $\tau(\S) = k$ is
\begin{align*}
P(E) & = \sum_{T \subseteq  \{1, \ldots, n\}, |T| = k}  P(C_{T} \cap D_{T})  = \binom{n}{k} (1-kp)^{k-1} (kp)^{n-k}. 
\end{align*}
\end{proof}

\subsection{Proofs of special cases} \label{section:specialcases}

The proof of \Cref{lemma:disjointkPiercingSets} relies on the fact that a fixed-length circular society $\S$ with $p < \frac{1}{2k}$ and $\tau(\S) = k$ is equivalent to a linear society.  If this is not the case, then $\S$ might not have a unique collection of $k$ disjoint approval sets that contain the left endpoints of all other approval sets.
This complicates the task of finding the probability distribution of $\tau(\S)$ when $p \geq \frac{1}{2k}$. However, when a society $S$ has a small number of voters we can determine the probability of each possible piercing number regardless of approval set length. 


For example, when a society has  $n = 2$ voters, \Cref{pnumless} and \Cref{randomsociety_tauk} allow us to completely describe the distribution of piercing numbers for all possible approval set lengths $p$. We state this in the following corollary.

\begin{corollary}\label{prop:n2}
For a random fixed-length circular society $\S$ with $n=2$ voters and approval set length $p$,
    \begin{align*}
    P(\tau(\S) = 1) & = \begin{cases}
        2(1-p)p & \text{ if } p \in [0, 1/2), \\
        1 & \text{ if } p \in [1/2, 1],
        \end{cases}\\
    P(\tau(\S) = 2) & = \begin{cases}
        1-2p(1-p) & \text{ if } p \in [0, 1/2), \\
        0 & \text{ if } p \in [1/2, 1].
        \end{cases}
    \end{align*}
\end{corollary}

\begin{proof}

We start by computing the probability that $\tau(\S) = 1$.  By \Cref{pnumless}, when $n = 2$ and $p \geq 1/2$, the piercing number must be 1.  So, $P(\tau(\S) = 1) = 1$.  Next consider the case when $p < 1/2$.  We can apply \Cref{randomsociety_tauk} to compute the probability that the piercing number is $k = 1$: 
\[P(\tau(\S) = 1) = \binom{2}{1} (1-p)p = 2p(1-p). \]
Finally, if the piercing number is not 1, it must be 2.  So, $P(\tau(\S) = 2) = 1 - P(\tau(\S) = 1)$.
\end{proof}

As $n$ increases, however, \Cref{pnumless} and \Cref{randomsociety_tauk} only give us piercing number probabilities for either small or large approval set lengths. 
For example, restricting our attention to the probability that $\tau(\S) = 1$, \Cref{randomsociety_tauk} states that if $p < \frac{1}{2}$, then $P(\tau(\S) = 1) = np^{n-1}$.  \Cref{pnumless} states that $\tau(\S) = 1$ for any fixed-length circular society $S$ with approval set length $p\geq \frac{n-1}{n}$.   Thus, when a society consists of more than $n = 2$ voters, it remains to find $P(\tau(\S) = 1)$ when the approval set length is $\frac{1}{2} \le p < \frac{n-1}{n}$. 

\Cref{prop:n3tau1} gives $P(\tau(\S) = 1)$ for societies with $3$ voters and approval set length $\frac{1}{2} \leq p < \frac{2}{3}$. \Cref{fig:distTau1} then shows this probability as a function of the approval set length $p$. This is the last piece of the puzzle needed for a complete picture of the piercing number probabilities for a society with $3$ voters, which we state in \cref{prop:n3}.

\begin{lemma}\label{prop:n3tau1}
  For a random fixed-length circular society $S$ with $n=3$ voters and approval set length $\frac{1}{2} \le p < \frac{2}{3}$, $P(\tau(\S) = 1) = -9p^2 + 12p - 3$.
\end{lemma}

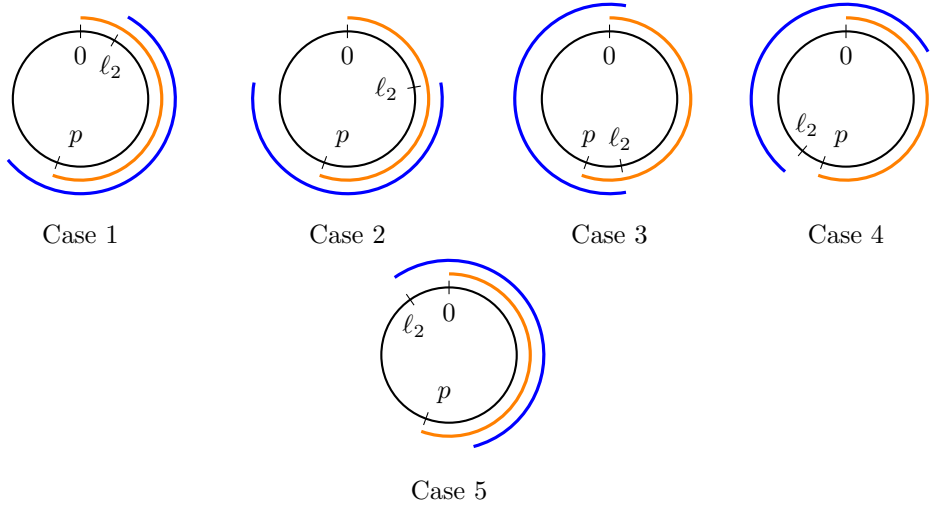
\begin{figure}[ht]
  \begin{center}
    \begin{tikzpicture}[scale=0.9]
      \draw[thick] (0,0) circle (1);
      \draw[orange, very thick] (90:1.2) arc (90:-110:1.2);
      \draw[blue, very thick] (60:1.4) arc (60:-140:1.4);
      \draw (0,0.9) node[below] {$0$} -- (0,1.1);
      \draw (-110:0.9) node[above right] {$p$} -- (-110:1.1);
      \draw (60:0.9) node [below] {$\ell_2$} -- (60:1.1);
      \node at (0,-2) {Case 1};
    \end{tikzpicture}
    \hspace{15pt}
    \begin{tikzpicture}[scale=0.9]
      \draw[thick] (0,0) circle (1);
      \draw[orange, very thick] (90:1.2) arc (90:-110:1.2);
      \draw[blue, very thick] (10:1.4) arc (10:-190:1.4);
      \draw (0,0.9) node[below] {$0$} -- (0,1.1);
      \draw (-110:0.9) node[above right] {$p$} -- (-110:1.1);
      \draw (10:0.9) node [left] {$\ell_2$} -- (10:1.1);
      \node at (0,-2) {Case 2};
    \end{tikzpicture}
    \hspace{15pt}
    \begin{tikzpicture}[scale=0.9]
      \draw[thick] (0,0) circle (1);
      \draw[orange, very thick] (90:1.2) arc (90:-110:1.2);
      \draw[blue, very thick] (80:1.4) arc (80:280:1.4);
      \draw (0,0.9) node[below] {$0$} -- (0,1.1);
      \draw (-110:0.9) node[above] {$p$} -- (-110:1.1);
      \draw (280:0.9) node [above] {$\ell_2$} -- (280:1.1);
      \node at (0,-2) {Case 3};
    \end{tikzpicture}    \hspace{15pt}
    \begin{tikzpicture}[scale=0.9]
      \draw[thick] (0,0) circle (1);
      \draw[orange, very thick] (90:1.2) arc (90:-110:1.2);
      \draw[blue, very thick] (30:1.4) arc (30:230:1.4);
      \draw (0,0.9) node[below] {$0$} -- (0,1.1);
      \draw (-110:0.9) node[above right] {$p$} -- (-110:1.1);
      \draw (230:0.9) node [above] {$\ell_2$} -- (230:1.1);
      \node at (0,-2) {Case 4};
    \end{tikzpicture}    \hspace{15pt}
    \begin{tikzpicture}[scale=0.9]
      \draw[thick] (0,0) circle (1);
      \draw[orange, very thick] (90:1.2) arc (90:-110:1.2);
      \draw[blue, very thick] (-75:1.4) arc (-75:125:1.4);
      \draw (0,0.9) node[below] {$0$} -- (0,1.1);
      \draw (-110:0.9) node[above right] {$p$} -- (-110:1.1);
      \draw (125:0.9) node [below] {$\ell_2$} -- (125:1.1);
      \node at (0,-2) {Case 5};
    \end{tikzpicture}  
  \end{center}
  \caption{The five cases in the proof of \Cref{prop:n3tau1}. $A_1 = [0,p]$ is drawn in orange, and $A_2 = [\ell_2,\ell_2+p]$ in blue. $A_3$ is not shown, but $A_3$ \emph{must} intersect $A_1 \cap A_2$ in Cases 1, 3, and 5;  $A_3$ \emph{might} intersect $A_1 \cap A_2$ in Cases 2 and 4.}
    \label{fig:n3tau1proof}
\end{figure}

\begin{proof}
  Without loss of generality, let $A_1 = [0,p]$.
  We consider five disjoint cases for the position of $A_2 = [\ell_2,\ell_2+p]$, as illustrated in \Cref{fig:n3tau1proof}.
  
  \begin{enumerate}[label=Case \arabic*:,left=20pt]
      \item Let $E_1$ be the event that $0 < \ell_2 < 2p-1$. In this case, $A_1 \cap A_2$ is an interval of length greater than $1-p$, so $A_3$ must intersect $A_1 \cap A_2$. Thus, $P((\tau = 1) \cap E_1) = P(E_1) = 2p-1$.
      \item Let $E_2$ be the event that $2p-1 < \ell_2 < 1-p$. In this case, $A_1 \cap A_2$ is an interval of length $p-\ell_2$, which is less than $1-p$. The probability that $E_2$ occurs and $A_3$ intersects $A_1 \cap A_2$ is
        \[ P((\tau = 1) \cap E_2) = \int_{2p-1}^{1-p} (2p-\ell_2)\ d\ell_2 = \frac{3}{2}\left( 2p - 3p^2 \right). \]
      \item Let $E_3$ be the event that $1-p < \ell_2 < p$. In this case, $A_1 \cap A_2$ is a union of two disconnected intervals: $A_1 \cap A_2 = [0,\ell_2+p-1] \cup [\ell_2, p]$. Since one of these intervals contains $0$ and the other contains $\frac{1}{2}$, $A_3$ \emph{must} intersect $A_1 \cap A_2$. Thus, $P((\tau = 1) \cap E_3) = P(E_3) = 2p-1$.
      \item Let $E_4$ be the event that $p < \ell_2 < 2-2p$. In this case, $A_1 \cap A_2$ is an interval of length $\ell_2 + p - 1$, which is less than $1-p$. The probability that $E_4$ occurs and $A_3$ intersects $A_1 \cap A_2$ is
        \[ P((\tau = 1) \cap E_4) = \int_{p}^{2-2p} (2p-1+\ell_2)\ d\ell_2 = \frac{3}{2}\left( 2p - 3p^2 \right). \]
      \item Let $E_5$ be the event that $2-2p < x < 1$. In this case, $A_1 \cap A_2$ is an interval of length greater than $1-p$. As in Case 1, $(A_1 \cap A_2) \cap A_3 \ne \emptyset$, so $P((\tau = 1) \cap E_5) = P(E_5) = 2p-1$.
  \end{enumerate}
  
  By the law of total probability, 
  \[ P(\tau = 1) = \sum_{i=1}^5 P((\tau = 1) \cap E_i) = 3(2p-1) + 2\cdot\frac{3}{2}\left( 2p - 3p^2 \right) = -9p^2 + 12p - 3. \]
\end{proof}

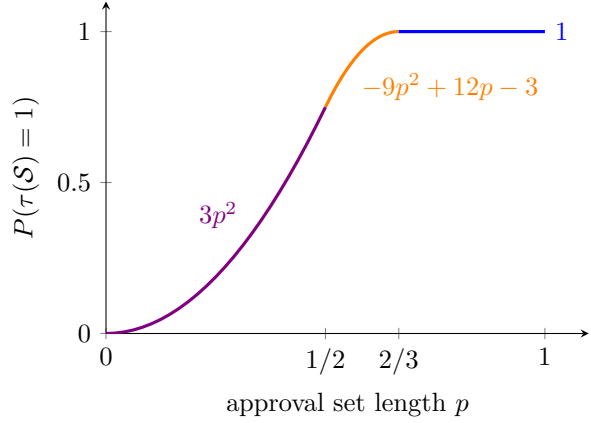
\begin{figure}[htb]
    \centering
    \begin{tikzpicture}
      \begin{axis}[width=8cm, height=6cm, axis y line = left, axis x line = bottom, xtick={0,0.5,0.667,1}, xticklabels={$0$,$1/2$,$2/3$,$1$}, xmin=0, xmax=1.1, xlabel={approval set length $p$}, ymin=0, ymax=1.1, ylabel={$P(\tau(\S)=1)$}]
        \addplot[color=violet, very thick, domain=0:0.5]{3*x*x} node[above left, pos=0.5] {$3p^2$};
        \addplot[color=orange, very thick, domain=0.5:0.667]{-9*\x*\x + 12*\x - 3} node[below right, pos=0.5] {$-9p^2 +12p -3$};
        \addplot[color=blue, very thick, domain=0.667:1]{1} node[right] {$1$};
      \end{axis}
    \end{tikzpicture}
    \caption{For fixed-length circular societies with $n=3$ voters, this plot shows $P(\tau(\S) = 1)$ as a function of the approval set length $p$. The purple curve for $0 < p < \frac{1}{2}$ is given by \Cref{randomsociety_tauk}, the orange curve for $\frac{1}{2} \le p < \frac{2}{3}$ is given by \Cref{prop:n3tau1}, and the blue line for $\frac{2}{3} \le p < 1$ by \Cref{pnumless}.}
    \label{fig:distTau1}
\end{figure}

\begin{corollary}\label{prop:n3}
For a random fixed-length circular society $\S$ with $n=3$,
    \begin{align*}
    P(\tau(\S) = 1) & = \begin{cases}
        3p^2  & \text{ if }p \in [0, 1/2), \\
        -9p^2 + 12p - 3  & \text{ if }p \in [1/2, 2/3), \\
        1  & \text{ if }p \in [2/3, 1],
        \end{cases}\\
    P(\tau(\S) = 2) & = \begin{cases}
    6p - 12 p^2  & \text{ if }p \in [0, 1/3), \\
    1 - 3p^2  & \text{ if }p \in [1/3, 1/2), \\
    9p^2 - 12p + 4  & \text{ if }p \in [1/2, 2/3), \\
    0  & \text{ if }p \in [2/3, 1],
    \end{cases}\\
    P(\tau(\S) = 3) & = \begin{cases}
        (1-3p)^2  & \text{ if }p \in [0, 1/3), \\
        0  & \text{ if }p \in [1/3, 1].
        \end{cases}
    \end{align*}
\end{corollary}

\begin{proof}
We start by computing the probability that $\tau(\S) = 1$.  \Cref{pnumless} tells us that if $p \geq 2/3$, then $P(\tau(\S) = 1) = 1$.  \Cref{randomsociety_tauk} tells us that if $p < 1/2$, then $P(\tau(\S) = 1) = 3p^2$.  Therefore, along with \Cref{prop:n3tau1},
\begin{align*}
    P(\tau(\S) = 1) & = \begin{cases}
    3p^2  & \text{ if } p \in [0, 1/2), \\
    -9p^2 + 12p - 3  & \text{ if } p \in [1/2, 2/3), \\
    1 & \text{ if } p \in [2/3, 1].
    \end{cases}
\end{align*}

Next, we compute the probability that $\tau(\S) = 3$.  \Cref{pnumless} tells us that if $p \geq 1/3$, then $\tau(\S) \leq 2$; this means that $P(\tau(\S) = 3) = 0$ when $p \geq 1/3$.  If $p < 1/3$, then \Cref{thm:solomon} tells us that the probability that three arcs of length $p$ are pairwise disjoint (that is, that the piercing number is 3) is $(1-3p)^2$.  So,
\begin{align*}
    P(\tau(\S) = 3) & = \begin{cases}
    (1-3p)^2  & \text{ if }p \in [0, 1/3), \\
    0  & \text{ if }p \in [1/3, 1].
    \end{cases}
\end{align*}

Finally, $P(\tau(\S) = 2) = 1 - \left( P(\tau(\S) = 1) + P(\tau(\S) = 3) \right)$, and the result follows.  
\end{proof}

Generalizing \Cref{prop:n3tau1} to larger values of $n$ is tedious due to the large number of cases, as the approval sets can intersect in many different ways.
We leave this as a direction for future study.

\subsection{Simulating Societies}\label{section:simulationresults}

We can only compute the exact piercing number probability  for fixed-length circular societies with small approval set length $p$ or small number of voters $n$, however, we can use simulation to explore the probability distribution of piercing numbers of random fixed-length circular societies with larger $p$ and/or larger $n$.
This provides an empirical answer to \cref{question:tauk}.
Given a fixed-length circular society $\S$ with $n$ voters and approval set length $p$, we can estimate the probability that the piercing number is $k$ by randomly generating $N$ such societies and computing the piercing number of each.
The proportion of the $N$ circular societies whose piercing number is equal to $k$ is our estimate of $P(\tau(\S) = k)$.

\Cref{randomsociety_tauk} states that the probability $P(\tau(\S) = k)$ for a random fixed-length circular society with approval set length $p < \frac{1}{2k}$ is given by \Cref{probEq}.
Our simulations suggest that \Cref{probEq} may give the correct probability even when $p$ is somewhat larger than $\frac{1}{2k}$, at least for certain values of $k$.
\Cref{fig:simulation_4sets} compares the simulated probability with the value given by \Cref{probEq} for random fixed-length circular societies with $n=4$ voters, each $k \in \{1, 2, 3, 4\}$, and a range of approval set lengths $p \in (0, 0.8)$.
We see that for each $\tau = k$ the estimated probability agrees with \Cref{probEq} for $p < \frac{1}{2k}$, as it should.
However, for $\tau > 2$, these values agree for $p$ beyond the bound stated in \Cref{randomsociety_tauk}. For $\tau = 3$, the values seem to match well up to $p = \frac{1}{3}$, and for $\tau = 4$, they seem to match well up to $p = \frac{1}{4}$.
This leads us to the following conjecture, which is further supported by additional simulations for larger $n$ and $k$.

\begin{conjecture}\label{taukconjecture}
For a random fixed-length circular society $S$,
\[ P(\tau(\S) = k) = \binom{n}{k} (1-kp)^{k-1}(kp)^{n-k} \]
for $p$ in an interval $(0, \frac{1}{2k} + \delta)$ for some $\delta > 0$.
\end{conjecture}

\begin{figure}[ht]
    \centering
    \includegraphics[width=11cm]{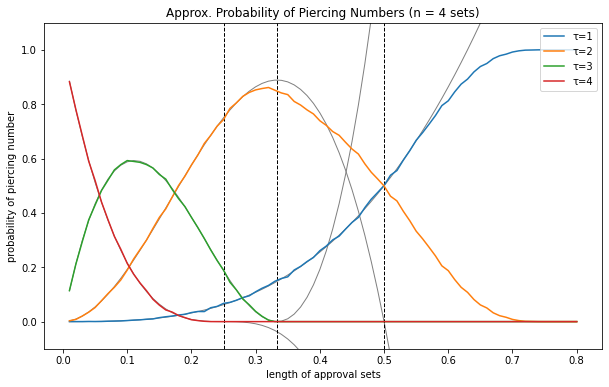}
    \caption{Comparison of simulated probabilities $P(\tau(\S)=k)$ with those given by \Cref{randomsociety_tauk} for each $k = 1, \ldots, 4$ in fixed-length circular societies with $n = 4$ voters and approval set lengths $0 \leq p \leq 0.8$.
    Simulated probabilities are plotted in color; polynomial curves given by \Cref{probEq} are plotted in gray.
    The simulated probabilities suggest that \Cref{probEq} gives the correct probability for $p$ somewhat larger than $\frac{1}{2k}$. The dashed vertical lines highlight $p=\frac{1}{4}$, $p=\frac{1}{3}$, and $p=\frac{1}{2}$. Note that when $\tau=3$ the formula seems to hold for $p \le \frac{1}{3}$, and when $\tau = 4$ it seems to hold for $p \le \frac{1}{4}$. }
    \label{fig:simulation_4sets}
\end{figure}

Furthermore, \Cref{table:simulation_results} displays estimated probabilities $P(\tau(\S) = k)$ for certain values of $p$ greater than $\frac{1}{2k}$.
For each combination of $n$, $k$, and $p$ listed in the table, we randomly generated $N=100{,}000$ fixed-length circular societies and computed the piercing number for each. The table shows that the simulated probabilities $P(\tau(\S)=k)$ are close to the values given by \Cref{probEq} and \Cref{taukconjecture}. Hence, there is evidence that the upper bound on $p$ should be greater than $\frac{1}{2k}$. However, we do not yet know the exact upper bound on $p$ at which \Cref{taukconjecture} fails to hold. We state this as a open question.

\begin{table}[!ht]
  \begin{center}
    \begin{tabulary}{3.3in}{RRRRR}
      \toprule
      $n$ & $p$ & \ $k$\ & estimated $P(\tau(\S) = k)$ & probability given by \Cref{taukconjecture} \\
      \midrule
      5 & 0.15 & 4 & 0.1903 & 0.1920 \\
      5 & 0.15 & 5 & 0.0038 & 0.0039 \\
      \midrule
      8 & 0.12 & 5 & 0.3102 & 0.3040 \\
      8 & 0.12 & 6 & 0.0254 & 0.0250 \\
      8 & 0.12 & 7 & 0.0001 & 0.0001 \\
      \midrule
      10 & 0.10 & 5 & 0.4896 & 0.4922 \\
      10 & 0.10 & 6 & 0.2808 & 0.2787 \\
      10 & 0.10 & 7 & 0.0309 & 0.0300 \\
      10 & 0.10 & 8 & 0.0004 & 0.0004 \\
      \bottomrule
    \end{tabulary}
  \end{center}
  \caption{Estimated piercing number probabilities from simulation compared with   probabilities given by \Cref{taukconjecture}. Results are rounded to four decimal places; simulated results of zero are not shown.}
  
  \label{table:simulation_results}
\end{table}

\begin{question}\label{question:bound_on_p}
What is the largest value of $\delta$ for which \Cref{taukconjecture} gives the probability that $\tau(\S) = k$? What is the relationship between $\delta$, the number of voters $n$, and the piercing number $k$?
\end{question}

Our probabilistic approach also helps us understand the ``average-case scenario'' for piercing numbers of fixed-length circular societies. If $\S$ is a fixed-length circular society with $n$ voters and approval set length  $p < \frac{1}{2n}$, then \Cref{randomsociety_tauk} holds for all $k \in \{1, \ldots, n\}$, and the expected value of $\tau(\S)$ is:
\begin{equation}\label{expValEq}
  E[ \tau(\S) ] = \sum_{k = 1}^{n} k \binom{n}{k} (1-kp)^{k-1}(kp)^{n-k}.
\end{equation}
However, given \Cref{taukconjecture}, the formula above may hold for $p$ somewhat larger than $\frac{1}{2n}$.
We simulated the average piercing number in random fixed-length circular societies with $p \ge \frac{1}{2n}$ and provide selected results  in \Cref{table:expectedvalues}.
We see that that the estimated average piercing number is close to the value given by \Cref{expValEq}.

\begin{table}[!ht]
    \centering
    \begin{tabulary}{3.5in}{RRRR} 
    \toprule
      $n$ &     $p$ &  estimated average piercing number &  expected value given by \Cref{expValEq} \\
    \midrule
      5 &  0.100 &   3.494 &    3.492 \\
      5 &  0.200 &   2.633 &    2.632 \\
      5 &  0.250 &   2.340 &    2.324 \\
    \midrule
     25 &  0.050 &   11.207 &    11.222 \\
     25 &  0.100 &   7.198 &     7.201 \\
     25 &  0.125 &   6.099 &     5.986 \\
    \midrule
     45 &  0.020 &    23.823 &   23.816 \\
     45 &  0.050 &    13.916 &   13.912 \\
     45 &  0.080 &     9.822 &    9.816 \\
    \bottomrule
    \end{tabulary}
    \caption{For each combination of values of $n$ and $p$, we simulated the average piercing number by randomly generating $N = 10{,}000$ fixed-length circular societies. We compare the simulation averages with the conjectured expected values, computed from \Cref{expValEq}.}
    \label{table:expectedvalues}
\end{table}

Returning to the scenario given in the introduction, suppose that a 25-member committee requires all members to attend a training session. They ask each member to select a three-hour interval during which they are available. (Given a 24-hour clock, this yields $p = \frac{1}{8}$).
\Cref{pnumless} tells us that as a worst case, the committee would need to schedule eight training sessions at different times to ensure that all members can attend one of the sessions. 
Assuming availability intervals are uniformly distributed around the circle, the average piercing number is about 6, so in an average case, the committee would only need to schedule 6 sessions in order to satisfy each member's preference. One the other hand, the uniform distribution is itself a sort of worst-case, as it spreads intervals uniformly around the clock. Realistically, committee members'  availabilities are likely to be clustered around certain times of the day; hence, a uniform distribution may not be the best model for this situation.  Consideration of other distributions is an interesting area for future work.

We invite  readers to  explore our simulations of piercing probabilities and average values using our circular society visualization and simulation tool, which is available at \url{https://github.com/tiasondjaja/circular_societies}.


\bibliographystyle{amsplain}


\end{document}

%% file: possible_piercing_diagram.tex
    \begin{tikzpicture}[scale=1.2]
        \draw[->] (0.5,0) node[left] {$0$} -- (9,0);
        \draw[->] (0.5,4) node[left] {$1$} -- (9,4);
        \draw (0.5,0) -- (0.5,4);
        \node at (0.2,2) {$p$};
        
        \node[below] at (9,0) {$n$};
        \foreach \i in {1, ..., 8} {
            \node[below] at (\i,0) {$\i$};
        }
        
        \draw[gray] (1.7,2) node[left] {\small $\frac{1}{2}$} -- (2.3,2);
        \draw[gray] (2.7,2.67) node[left] {\small $\frac{2}{3}$} -- (3.3,2.67);
        \draw[gray] (3.7,3) node[left] {\small $\frac{3}{4}$} -- (4.3,3);
        \draw[gray] (4.7,3.2) node[left] {\small $\frac{4}{5}$} -- (5.3,3.2);
        \draw[gray] (5.7,3.33) node[left] {\small $\frac{5}{6}$} -- (6.3,3.33);
        \draw[gray] (6.7,3.43) node[left] {\small $\frac{6}{7}$} -- (7.3,3.43);
        \draw[gray] (7.7,3.5) node[left] {\small $\frac{7}{8}$} -- (8.3,3.5);
        
        \draw[gray] (2.7,1.33) node[left] {\small $\frac{1}{3}$} -- (4.3,1.33);
        \draw[gray] (4.7,1.6) node[left] {\small $\frac{2}{5}$} -- (6.3,1.6);
        \draw[gray] (6.7,1.71) node[left] {\small $\frac{3}{7}$} -- (8.3,1.71);
        
        \draw[gray] (3.7,1) node[left] {\small $\frac{1}{4}$} -- (6.3,1);
        \draw[gray,->] (6.7,1.14) node[left] {\small $\frac{2}{7}$} -- (9,1.14);
        
        \draw[gray] (4.8,0.8) node[left] {\scriptsize $\frac{1}{5}$} -- (8.2,0.8);
        
        \draw[gray, ->] (5.8,0.67) node[below left=-0.2 and -0.1] {\scriptsize $\frac{1}{6}$} -- (9,0.67);
        
        \draw[gray, ->] (6.8,0.57) node[below left=-0.1 and -0.1] {\scriptsize $\frac{1}{7}$} -- (9,0.57);
        
        \draw[gray, ->] (7.8,0.5) node[below left=-0.1 and -0.1] {\scriptsize $\frac{1}{8}$} -- (9,0.5);

        \draw[red, line width=2pt] (1,0) -- (1,4);
        \fill[red] (1,4) circle (2.3pt);
        \fill[red] (1,0) circle (2.3pt);
        
        \draw[red, line width=2pt] (2,2) -- (2,4);
        \draw[orange, line width=2pt] (2,0) -- (2,2);
        \fill[red] (2,4) circle (2.3pt);
        \fill[red] (2,2) circle (2.3pt);
        \fill[orange] (2,0) circle (2.3pt);
        
        \draw[red, line width=2pt] (3,2.67) -- (3,4);
        \draw[orange, line width=2pt] (3,1.33) -- (3,2.67);
        \draw[yellow!60!brown, line width=2pt] (3,0) -- (3,1.33);
        \fill[red] (3,4) circle (2.3pt);
        \fill[red] (3,2.67) circle (2.3pt);
        \fill[orange] (3,1.33) circle (2.3pt);
        \fill[yellow!60!brown] (3,0) circle (2.3pt);
        
        \draw[red, line width=2pt] (4,3) -- (4,4);
        \draw[orange, line width=2pt] (4,1.33) -- (4,3);
        \draw[yellow!60!brown, line width=2pt] (4,1) -- (4,1.33);
        \draw[green, line width=2pt] (4,0) -- (4,1);
        \fill[red] (4,4) circle (2.3pt);
        \fill[red] (4,3) circle (2.3pt);
        \fill[orange] (4,1.33) circle (2.3pt);
        \fill[yellow!60!brown] (4,1) circle (2.3pt);
        \fill[green] (4,0) circle (2.3pt);
        
        \draw[red, line width=2pt] (5,3.2) -- (5,4);
        \draw[orange, line width=2pt] (5,1.6) -- (5,3.2);
        \draw[yellow!60!brown, line width=2pt] (5,1) -- (5,1.6);
        \draw[green, line width=2pt] (5,0.8) -- (5,1);
        \draw[blue, line width=2pt] (5,0) -- (5,0.8);
        \fill[red] (5,4) circle (2.3pt);
        \fill[red] (5,3.2) circle (2.3pt);
        \fill[orange] (5,1.6) circle (2.3pt);
        \fill[yellow!60!brown] (5,1) circle (2.3pt);
        \fill[green] (5,0.8) circle (2.3pt);
        \fill[blue] (5,0) circle (2.3pt);
        
        \draw[red, line width=2pt] (6,3.33) -- (6,4);
        \draw[orange, line width=2pt] (6,1.6) -- (6,3.33);
        \draw[yellow!60!brown, line width=2pt] (6,1) -- (6,1.6);
        \draw[green, line width=2pt] (6,0.8) -- (6,1);
        \draw[blue, line width=2pt] (6,0.67) -- (6,0.8);
        \draw[cyan, line width=2pt] (6,0) -- (6,0.67);
        \fill[red] (6,4) circle (2.3pt);
        \fill[red] (6,3.33) circle (2.3pt);
        \fill[orange] (6,1.6) circle (2.3pt);
        \fill[yellow!60!brown] (6,1) circle (2.3pt);
        \fill[green] (6,0.8) circle (2.3pt);
        \fill[blue] (6,0.67) circle (2.3pt);
        \fill[cyan] (6,0) circle (2.3pt);
        
        \draw[red, line width=2pt] (7,3.43) -- (7,4);
        \draw[orange, line width=2pt] (7,1.71) -- (7,3.43);
        \draw[yellow!60!brown, line width=2pt] (7,1.14) -- (7,1.71);
        \draw[green, line width=2pt] (7,0.8) -- (7,1.14);
        \draw[blue, line width=2pt] (7,0.67) -- (7,0.8);
        \draw[cyan, line width=2pt] (7,0.57) -- (7,0.67);
        \draw[violet, line width=2pt] (7,0) -- (7,0.57);
        \fill[red] (7,4) circle (2.3pt);
        \fill[red] (7,3.43) circle (2.3pt);
        \fill[orange] (7,1.71) circle (2.3pt);
        \fill[yellow!60!brown] (7,1.14) circle (2.3pt);
        \fill[green] (7,0.8) circle (2.3pt);
        \fill[blue] (7,0.67) circle (2.3pt);
        \fill[cyan] (7,0.57) circle (2.3pt);
        \fill[violet] (7,0) circle (2.3pt);
        
        \draw[red, line width=2pt] (8,3.5) -- (8,4);
        \draw[orange, line width=2pt] (8,1.71) -- (8,3.5);
        \draw[yellow!60!brown, line width=2pt] (8,1.14) -- (8,1.71);
        \draw[green, line width=2pt] (8,0.8) -- (8,1.14);
        \draw[blue, line width=2pt] (8,0.67) -- (8,0.8);
        \draw[cyan, line width=2pt] (8,0.57) -- (8,0.67);
        \draw[violet, line width=2pt] (8,0.5) -- (8,0.57);
        \draw[magenta, line width=2pt] (8,0) -- (8,0.5);
        \fill[red] (8,4) circle (2.3pt);
        \fill[red] (8,3.5) circle (2.3pt);
        \fill[orange] (8,1.71) circle (2.3pt);
        \fill[yellow!60!brown] (8,1.14) circle (2.3pt);
        \fill[green] (8,0.8) circle (2.3pt);
        \fill[blue] (8,0.67) circle (2.3pt);
        \fill[cyan] (8,0.57) circle (2.3pt);
        \fill[violet] (8,0.5) circle (2.3pt);
        \fill[magenta] (8,0) circle (2.3pt);
        
        
        \fill[red] (10,3.4) rectangle ++(0.2,0.2);
        \node[red,right] at (10.3,3.5) {$\tau = 1$};
        \fill[orange] (10,3) rectangle ++(0.2,0.2);
        \node[orange,right] at (10.3,3.1) {$\tau \le 2$};
        \fill[yellow!30!brown] (10,2.6) rectangle ++(0.2,0.2);
        \node[yellow!30!brown,right] at (10.3,2.7) {$\tau \le 3$};
        \fill[green] (10,2.2) rectangle ++(0.2,0.2);
        \node[green,right] at (10.3,2.3) {$\tau \le 4$};
        \fill[blue] (10,1.8) rectangle ++(0.2,0.2);
        \node[blue,right] at (10.3,1.9) {$\tau \le 5$};
        \fill[cyan] (10,1.4) rectangle ++(0.2,0.2);
        \node[cyan,right] at (10.3,1.5) {$\tau \le 6$};
        \fill[violet] (10,1) rectangle ++(0.2,0.2);
        \node[violet,right] at (10.3,1.1) {$\tau \le 7$};
        \fill[magenta] (10,0.6) rectangle ++(0.2,0.2);
        \node[magenta,right] at (10.3,0.7) {$\tau \le 8$};
        
      \end{tikzpicture}